\documentclass[10pt,a4paper,twoside]{amsart} 
\usepackage[english]{babel}
\usepackage[utf8]{inputenc}
\usepackage[pdftex]{graphicx}
\usepackage{hyperref}
\usepackage{subcaption}
\usepackage{amssymb}
\usepackage{adjustbox}
\usepackage{amsmath}
\usepackage[noabbrev,capitalize]{cleveref}
\usepackage{latexsym}
\usepackage{amsthm}
\usepackage{verbatim}
\usepackage{tikz-cd}
\usepackage{enumitem}
\usepackage{afterpage}
\usepackage{xcolor,colortbl,soul}   
\usepackage{tabularx}
\usepackage{float}
\counterwithin{figure}{section}
\newcolumntype{P}[1]{>{\centering\arraybackslash}p{#1}}
\usepackage{chngpage}
\usepackage{geometry}

\def\vertexsize {1.2pt}   
\newcommand{\vertex}[2][1]{\fill (#2) circle [radius = #1 * \vertexsize];}
\newcommand{\divisor}[2][1]{\fill [red] (#2) circle [radius = 2* \vertexsize];}

\usepackage{xcolor,color,soul}

\definecolor{verde}{rgb}{0.01, 0.75, 0.24}
\newcommand{\val}{\operatorname{val}}
\newcommand{\rr}{\operatorname{r}}
\newcommand{\g}{\operatorname{g}}

\newcommand{\ddeg}{\operatorname{deg}}

\newcommand{\R}{\mathbb{R}}
\newcommand{\ZZ}{\mathbb{Z}}
\newcommand{\ga}{\gamma}
\newcommand{\Ga}{\Gamma}
\DeclareMathOperator{\ddiv}{div}

\theoremstyle{plain}                       
\newtheorem{theo}{Theorem}[section]
\newtheorem{maintheorem}{Theorem}	
 
\newtheorem{conj}{Conjecture}    
\newtheorem{prop}[theo]{Proposition}    
       
\newtheorem{lem}[theo]{Lemma} 
\theoremstyle{definition}               
\newtheorem{defin}{Definition}

\newtheorem{ex}{Example}   
\newtheorem{rk}{Remark}   
\theoremstyle{remark} 

\title{A tropical version of Martens' theorem for metric graphs}
\author{Giusi Capobianco}
\address{Department of Mathematics\\ University of Roma Tor Vergata \\ 00133 Rome, Italy}
\email{\href{mailto:capobianco@axp.mat.uniroma2.it}{capobianco@axp.mat.uniroma2.it}}

\author{Angelina Zheng}
\address{Department of Mathematics\\ University of T\"ubingen \\ 72076 T\"ubingen, Germany}
\email{\href{mailto:zheng@math.uni-tuebingen.de}{zheng@math.uni-tuebingen.de}}
\date{}
\thanks{Both authors are
members of the INDAM group GNSAGA. The second author is supported by Alexander von Humboldt Foundation.}
\begin{document}
\begin{abstract}
    We study the conjecture stated by Jensen and Len on a tropical version on Martens' theorem via the Brill--Noether rank of a tropical curve. 
    We recall Coppens' counterexample of Martens-special chain of cycles, and we generalize the construction defining another class of graphs, Martens-special trees of cycles, for which the conjecture does not hold in a similar setting. 
    These are not the only counterexamples. 
    However, we prove that the conjecture holds for all metric graphs with a stricter assumption on the degree in the Brill--Noether rank.
\end{abstract}
\thanks{}
\maketitle
\tableofcontents
\section{Introduction}

The classical Martens' theorem provides insight into the structure and properties of algebraic curves characterizing hyperelliptic curves and refining the result obtained by Clifford \cite{Clifford}, \cite[Chapter III]{ACGH}, as it offers an understanding of the Brill--Noether loci. 
This result determines their dimensions and plays a significant role in various geometric constructions and classifications in Brill--Noether theory. 
Let us recall Martens' theorem. 

\begin{theo}{\cite[Theorem 5.1]{ACGH}}
    Let $C$ be a smooth curve of genus $g,$ and $d,r$ integers satisfying $0<2r\leq d<g.$ Then
    $\operatorname{dim}W_d^r(C)\leq d-2r,$ and equality holds precisely when $C$ is hyperelliptic.
\end{theo}

In tropical geometry, the Brill--Noether theory has been adapted to study the combinatorial analogs of these algebraic objects. The tropical version of Brill--Noether theory translates problems about divisors on algebraic curves into problems about divisors on metric graphs.
The divisor theory on graphs and metric graphs was developed by Baker and Norine in \cite{BN09}, \cite{BC} and later studied by many other authors.
One of the most useful result is Baker's specialization Lemma,  extended by Caporaso \cite{Cap12}, which gives a relation between the rank of a divisor and that of its tropicalization, namely the rank of a divisor can only increase under the specialization map. 
This relation between curves and graphs provides alternative proofs for classical algebraic results, such as the Brill--Noether theorem (\cite{CDPR}).

In the past decades tropical analogs of well-known algebraic theorems such as the Riemann--Roch theorem and Clifford's theorem, have been proved by various authors; see, for instance, \cite{GK08}, \cite{AC}, \cite{MikZhark}, \cite{facchini}, \cite{Coppens_Clifford}.
However, it is not always true that the same statements hold in tropical geometry. An example is indeed Martens' theorem: the naive tropical analog of the above theorem does not hold and the known statement for abstract tropical curves is the following. 

\begin{theo}{\cite[Theorem A.1]{L}}
    Let $\Gamma$ be a metric graph of genus $g,$ and $d,r$ integers satisfying $0<2r\leq d<g.$ Then 
    $\operatorname{dim}W_d^r(\Gamma)\leq d-2r$ and equality holds when $\Gamma$ is hyperelliptic. Moreover, equality may also hold for $\Gamma$ non-hyperelliptic.
\end{theo}
In \cref{ex.counterexnonhyp} we will present, for instance, a non-hyperelliptic metric graph for which $\operatorname{dim}W_d^r(\Gamma)= d-2r.$\sloppy

Jensen and Len suggested in \cite{L} that, in order to have a tropical analog of Martens' theorem, one should replace the dimension of the Brill--Noether locus with the \emph{Brill--Noether rank}, first introduced in \cite{LPP}.
The Brill--Noether rank $w_d^r(\Gamma)$ of a tropical curve $\Gamma$, contrary to the dimension of the Brill--Noether locus $\operatorname{dim}W_d^r(\Gamma)$, varies upper semi-continuously on the moduli space of tropical curves, see \cite[Theorem 5.4]{Len} and \cite[Theorem 1.2, 1.6 and 1.7]{LPP}. Therefore, it might be a good candidate for the analog of Martens' theorem for tropical curves. 

\begin{prop}{\cite[Propostion A.3]{L}}\label{prp:ineq}
    Let $\Ga$ be a metric graph of genus $g$, and let $r, d$ be as in the conditions of Martens’ theorem. Then
$$w_d^r(\Gamma)\leq d-2r.$$
Moreover, for hyperelliptic graphs, $w_d^r(\Gamma)= d-2r.$
\end{prop}

Furthermore, Jensen and Len have conjectured that equality holds only for hyperelliptic tropical curves, suggesting the following tropical version of Martens' theorem.
\begin{conj}{\cite{L}}\label{JLconj}
    Let $\Gamma$ be a metric graph of genus $g,$ and $d,r$ as above. Then 
    $w_d^r(\Gamma)\leq d-2r,$ and equality holds precisely when $\Gamma$ is hyperelliptic.
\end{conj}

Recently, Coppens has proved that the conjecture holds for chains of cycles of genus $g$ for rank $r$ and degree $d$ such that $1\leq r\leq g-2$ and $2r\leq d\leq g-3+r$. 

\begin{theo}{\cite[Theorem A]{coppens2024studyhmartenstheorem}}\label{th:Copp1}
    Let $\Ga$ be a chain of cycles of genus $g$. Let $r$ be an integer with $1\leq r\leq g-2$ and let $d$ be an integer with $2r\leq d\leq g-3+r$. Then $w^r_d(\Ga)=d-2r$ implies $\Ga$ hyperelliptic.
\end{theo}
 
However, Conjecture \ref{JLconj} is not true in its full generality. In \cite{coppens2024studyhmartenstheorem}, the author provides counterexamples to such a conjecture proving the existence of graphs of genus $g\geq 2r+3$ called \emph{Martens-special} chains of cycles which are not hyperelliptic and satisfy the equality $w^{r}_{g-2+r}(\Ga)=g-2-r$.

\begin{theo}{\cite[Theorem B]{coppens2024studyhmartenstheorem}}\label{th:Copp2}
A chain of cycles $\Ga$ satisfies $w^r_{g-2+r}(\Ga)=g-2-r$ for some integer $1\leq r \leq g-2$ if and only if either $\Ga$ is hyperelliptic or $\Ga$ is a Martens-special chain of cycles of rank $r$.
\end{theo}

We further investigate this conjecture by first generalizing Coppens' counterexample of Martens-special chains of cycles to a class of graphs that we will call \emph{Martens-special trees of cycles}.
More precisely, for these graphs, we first prove the analog statements of Theorems \ref{th:Copp1} and \ref{th:Copp2}.
Moreover, we will show that these are not the only graphs for which Conjecture \ref{JLconj} fails. 
Indeed, we will see that also graphs that are obtained by changing the length of some edge in a hyperelliptic graph represent also a counterexample.
However, if we impose a stricter bound on the degree then Conjecture \ref{JLconj} is true.
The precise statement of our main theorem is the following.

\begin{maintheorem}\label{th:main_theo}
    Let $\Gamma$ be a metric graph of genus $g,$ and $d,r$ integers with $0<2r\leq d\leq g-3+r$.  Then 
    $w_d^r(\Gamma)\leq d-2r.$ Equality holds precisely when $\Gamma$ is hyperelliptic.
\end{maintheorem}

This in particular provides an alternative proof of the tropical Clifford's theorem when $d,r,g$ are positive integers satisfying $2r\leq d\leq g-3+r$.

We will first recall in Section \ref{sc:div_theory} the notions and definitions of divisor theory on metric graphs. 
In Section \ref{sc:Coppens}, we will present Coppens' counterexample for Conjecture \ref{JLconj} and construct  new classes of graphs for which the conjecture is again not true. Finally, in Section \ref{sc:main_proof} we will prove that, restricting to $d\leq g-3+r$, the conjecture holds.

\medskip
\noindent 
{\bf Acknowledgements.} We thank Yoav Len and Margarida Melo for their insightful comments and suggestions.
We are particularly grateful to Marc Coppens for his detailed review and extensive feedback on earlier versions of the article, as well as his substantial contributions throughout its development.
\section{Divisor theory on metric graphs}\label{sc:div_theory}
In this section, we recall the basic definitions and properties of divisor theory on metric graphs that we need throughout the article. For more details, see \cite{BN09}, \cite{BF11}, \cite{BC}, and \cite{Len_Survey}.

A \emph{metric graph} is a metric space $\Ga$ such that there exist a graph $G$ and a length function on its set of edges $l\colon E(G)\to \R_{>0}$ such that $\Ga$ is obtained from $(G,l)$ by gluing intervals $[0,l(e)]$ for  $e\in E(G)$ at their endpoints, as prescribed by the combinatorial data of $G$.
We say that $(G,l)$ is a \emph{model} for $\Ga$ and the \emph{genus} of $\Ga$ is the first Betti number of $G$. Notice that the genus is independent of the choice of the model.
We will only consider finite, connected and weightless graphs.

The \emph{distance} $d(x, y)$ between two points $x,y\in\Ga$ is the length of the shortest path between them. 
For any point $x$, its \emph{valence}, denoted  $\val(x)$, is the number of connected components in $U_x\setminus\{x\}$, where $U_x$ is an arbitrarily small open neighborhood containing $x$.   

The \emph{canonical model} $(G_0,l_0)$ of $\Ga$ is the model obtained by taking as its set of vertices
$
V(G_0)=\{x\in\Ga|\,\mathrm{val}(x)\geq 3 \}.
$
Unless otherwise specified, we will work with the canonical model of a metric graph. In particular, we say that $G_0$ is the \emph{underlying combinatorial graph} of $\Ga$.

Moreover, let us recall that a metric graph is said to be \emph{$k$-edge connected} if the removal of any subset of $l\leq k-1$ edges in its canonical model does not disconnect the graph. A \emph{$k$-edge cut} is a set of $k$ edges, in the canonical model, whose removal disconnects the graph. All incident edges to a vertex clearly form a $k$-edge cut, this will be called \emph{trivial}.
 
A \emph{divisor} $D$ is a formal sum $D=\sum_{p\in \Gamma}a_pp,$ with $a_p\in\ZZ$ and $a_p\neq0$ for finitely many $p\in \Ga$.
Its \emph{degree} is $\deg(D)=\sum_{p\in\Gamma}a_p$ and the \emph{group of divisors} of $\Gamma$ is $\mathrm{Div}(\Gamma),$ 
and it decomposes as 
$\operatorname{Div}(\Gamma)=\cup_{d\in \mathbb Z} \operatorname{Div}^d(\Gamma),$
where $\operatorname{Div}^d(\Gamma)$ is the group of divisors of degree $d$.
A divisor $D=\sum_{p\in \Gamma}a_pp $ is called \emph{effective}, and we write $D\geq 0,$ if 
$a_p\geq 0$, for all $p\in \Ga$.
We also define the \emph{support} of $D$ as the set $\operatorname{Supp}(D)=\{p\in\Ga|a_p\neq 0\}.$ 

A \emph{rational function} on $\Gamma$ is a continuous, piecewise linear function $f\colon \Gamma\to \R$  whose slopes are all integers. 
The \emph{principal divisor} associated to the rational function $f$ is
	$
	\ddiv(f)=\sum_{P\in \Gamma}\sigma_p(f)p,
	$
	 where $\sigma_p(f)$ is the sum of the slopes of $\Gamma$ in all directions emanating from $p$. 
Two divisor $D_1,D_2$ are \emph{linearly equivalent}, and we write $D_1\sim D_2$, if their difference is a principal divisor.

The (\emph{Baker and Norine}) \emph{rank} of a divisor $D$ on $\Gamma$ is
	  \begin{equation*}
	  \begin{aligned}
	  &\rr(D)=\max\{k\in \ZZ: \forall\,E\in\mathrm{Div}^k(\Gamma); E\geq 0, \exists E'\geq 0 \text{ such that } D-E\sim E' \}, 
	  \end{aligned}
	  \end{equation*}
       or $\operatorname{r}(D)=-1$ if the above set is empty.

From \cite[Corollary]{BN09}, the contraction of bridges, i.e. $1$-edge cuts, does not change the rank. From now on we will only consider metric graphs without points of valence $1$  (for metric graphs with leaves, results on their Brill--Noether loci are the same).

Let us recall the Riemann--Roch theorem for metric graphs, first proved for combinatorial graphs in \cite{BN07}, that will be used throughout the paper. 

\begin{theo}{\cite{GK08,MikZhark}}
Let $D\in \operatorname{Div}(\Ga),$ and $K$ be the canonical divisor of $\Ga.$ Then
    $$\rr(D)-\rr(K-D) = \ddeg(D) + 1 -\g(\Ga).$$
\end{theo}

It is an immediate consequence of the  Riemann--Roch theorem that if $\ddeg(D)\geq \g(\Ga)+1$, then $\rr(D)\geq 1$. 

In order to determine if the rank of a given divisor is greater or equal than $1,$ we will also make extensive use of Dhar's burning algorithm, \cite{D}.
Given an effective divisor $D\in\operatorname{Div}(\Gamma)$ and a point $w\in \Gamma$, Dhar's burning algorithm consists in starting a fire at $w,$ which spreads along each edge incident to $w$. Any $x\in \Gamma$ with $D(x)>0$ can control the fire in $D(x)$ directions: if $D(x)$ is bigger or equal than the number of burnt edges incident to $x$ then the fire stops at $x$, otherwise it burns $x$ and continues to spread along all remaining edges incident to $x$.
Following \cite{BS}, if the entire graph burns, then $D-w$ is a $w$-reduced divisor and in particular $\rr(D-w)=-1$, meaning that it cannot be linearly equivalent to an effective divisor, therefore $\operatorname{r}(D)<1$.

\begin{defin}
  For non-negative integers $r$ and $d$, the \emph{Brill--Noether locus}  $W^r_d(\Gamma)$ is the set of divisor classes of degree $d$ and rank at least $r$ on the metric graph $\Ga$.
\end{defin}

As already mentioned in the introduction, by \cite[Theorem A.1]{L}, unlike the algebraic case, it is not true that $\operatorname{dim}W_d^r=d-2r$ only for hyperelliptic graphs. Here is an example.

\begin{ex}\label{ex.counterexnonhyp}
    Figure \ref{fg:counterex1} shows a non-hyperelliptic metric graph with two divisors $x+y+z_i\in W_3^1(\Gamma),$ for $i=1,2.$
Moreover, one can check that any $x+y+z$; $z\in \overline{z_1z_2}$, is a divisor of degree $3$ and rank $1$, therefore $\operatorname{dim}W_3^1(\Gamma)= 3-2\cdot1.$

\begin{figure}[!ht]
\centering
\begin{tikzcd}
\begin{tikzpicture}
    
    \draw(-0.5,0)--(0,2);
    \vertex{-0.5,0}
    \vertex{2.5,0}
    \draw(0,1)--(-0.5,0);
    \draw(-0.5,0)--(2.5,0);
    \draw(0,2)--(2,2);
    \draw(0,1)--(2,1);
    \draw(0,2)--(0,1);
    \draw(2,2)--(2,1);
    \draw(2,2)--(2.5,0);
    \draw(2.5,0)--(2.2,0.6);
    \draw(2,1)--(2.5,0);
    \foreach \i in {0,2} {\foreach \j in {1,2} {\vertex{\i, \j}}}
    \divisor{2,0}
    \divisor{2,1}
    \divisor{2,2}
    \draw (2,2) node[anchor=south ] {$x$};
    \draw (2,1) node[anchor=north ] {$y$};
    \draw (2,0) node[anchor=north ] {$z_1$};
    \end{tikzpicture}&
    \begin{tikzpicture}
    
    \draw(-0.5,0)--(0,2);
    \vertex{-0.5,0}
    \vertex{2.5,0}
    \draw(0,1)--(-0.5,0);
    \draw(-0.5,0)--(2.5,0);
    \draw(0,2)--(2,2);
    \draw(0,1)--(2,1);
    \draw(0,2)--(0,1);
    \draw(2,2)--(2,1);
    \draw(2,2)--(2.5,0);
    \draw(2.5,0)--(2.2,0.6);
    \draw(2,1)--(2.5,0);
    \foreach \i in {0,2} {\foreach \j in {1,2} {\vertex{\i, \j}}}
    \draw (2,2) node[anchor=south ] {$x$};
    \draw (2,1) node[anchor=north ] {$y$};
    \draw (2.5,0) node[anchor=north ] {$z_2$};
    \divisor{2.5,0}
    \divisor{2,1}
\divisor{2,2}\end{tikzpicture}\end{tikzcd}\caption{A metric graph $\Gamma$ with two non-equivalent divisors of degree 3 and rank 1. Here $\operatorname{dim}W^1_3(\Gamma)=1.$}\label{fg:counterex1}\end{figure}
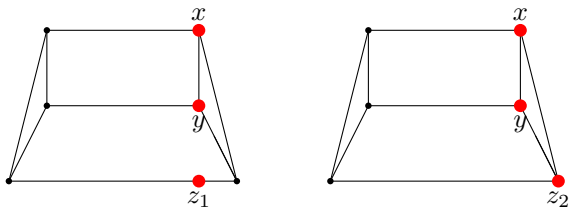
\end{ex}

The {Brill--Noether rank} is defined as follows.
\begin{defin}{\cite[Definition A.2]{L}} The \emph{Brill--Noether rank} of a metric graph $\Gamma$ is the integer$$w_d^r(\Gamma):=\operatorname{max}\{\rho; \text{ for any } E\in Div_{+}^{r+\rho}(\Gamma); \exists \,D \in W_d^r(\Gamma) \text{ such that }\operatorname{Supp}(E)\subseteq \operatorname{Supp}(D)\}.$$\end{defin}

\section{Counterexamples for Jensen--Len's conjecture }\label{sc:Coppens}
In this section we first recall the metric graphs studied by Coppens in \cite{coppens2024studyhmartenstheorem} and \cite{Copp2}, namely \emph{Martens-special chains of cycles}. 
Coppens proved that for such graphs Conjecture \ref{JLconj} does not always hold. 
Later, we will also provide other counterexamples, first by slightly modifying such graphs. 

\subsection{Martens-special chains of cycles}
Let us first recall the definition of Martens-special chains of cycles, as defined by Coppens.
Let $\Gamma$ be a chain of cycles of genus $g.$ Then $\Gamma$ has $g$ cycles $C_1,\dots,C_{g},$ where $C_1,C_g$ are loops and the cycles $C_2,\dots, C_{g-1}$ have vertices $v_2,w_2,\dots,v_{g-1},w_{g-1}$ as in Figure \ref{fg:chain_cycles}.

\begin{figure}[h!]
\centering
\begin{tikzcd}
\begin{tikzpicture}
    \draw (-0.5,0)circle (0.5);
    \draw (0,0)--(0.75,0);
    \draw (1.5,0) circle (0.75);
    \draw (2.25,0)--(3,0);
    \draw (3.5,0) circle (0.5);
    \draw (4,0)--(4.5,0);
    \draw (5,0) node {$\dots$};
    \draw(5.5,0)--(6,0);
    \draw (6.5,0)circle (0.5);
    \vertex{0.75,0}\vertex{2.25,0}
    \vertex{3,0}\vertex{4,0}
    \draw (-0.5,0) node {$C_1$};
    \draw (1.5,0) node {$C_2$};
    \draw (3.5,0) node {$C_3$};
    \draw (6.5,0) node {$C_g$};
    \draw (0.55,0) node[anchor=south]{$v_2$};
    \draw (2.5,0) node[anchor=north ] {$w_2$};
    \draw (2.8,0) node[anchor=south ] {$v_3$};
    \draw (4.2,0) node[anchor=north ] {$w_3$};
    \end{tikzpicture}\end{tikzcd}\caption{A chain of cycles of genus $g$.}\label{fg:chain_cycles}
\end{figure}
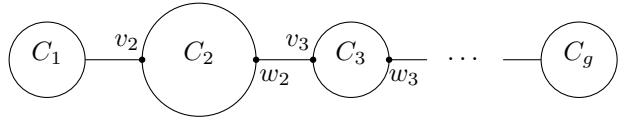
Recall that contraction of bridges does not change the rank, so we will assume that the bridges have arbitrary lengths (also zero).

For $i=2\dots, g-1,$ denote by $l_i$ the length of the $i$-th cycle and define the integers
\begin{equation}\label{eq:m}
m_i=
\begin{cases}
\frac{l_i}{d(v_i,w_i)}& \text{ if }\frac{l_i}{d(v_i,w_i)}\in \mathbb Z_{>0}, \\
0 &\text{otherwise.}
\end{cases}
\end{equation}

Let us observe that $m_i=2$ for any $i=2\dots, g-1,$ if and only if the chain of cycle is hyperelliptic, \cite[Section 3]{MC}. Moreover, we will simply write \emph{hyperelliptic cycle} when referring to a
 cycle $C_1,C_g$ or $C_i$ with $m_i=2.$ We will say that it is non-hyperelliptic otherwise. In what follows, we will consider non-hyperelliptic chains of cycles. We refer to \cite[Section 4]{coppens2024studyhmartenstheorem} for more details.

\begin{defin}\label{df:M_special}
    If $\Gamma$ is a chain of cycles of genus $g\geq 2r+3$ then it is \emph{Martens-special} of rank $r$ and type $k\geq1$ if there exist integers $j_1,\dots j_k$ such that 
    \begin{enumerate}
    \item $r+1<j_1<\dots<j_k<g-r;$
    \item $j_{i+1}-j_i\geq r+1$ for each $1\leq i\leq k-1$;
    \item $m_s\neq2$ if $s=j_1,\dots,j_k$ and $m_s=2$ if $s\neq j_1,\dots,j_k$.
    \end{enumerate}
\end{defin}

From the definition, one can think of a Martens-special chain of cycle of rank $r$ and type $k$ as a chain of cycles with $k$ non-hyperelliptic cycles, such that at least the first $r+1$ and the last $r+1$ cycles are hyperelliptic and between any two non-hyperelliptic cycles, there are at least $r$ hyperelliptic ones. 
Let us notice, moreover, that a Martens-special chain of rank $r,$ then satisfies the above conditions for any $r',$ with $1\leq r'\leq r.$ 

Coppens proved that on such graphs Conjecture \ref{JLconj} holds for $d\leq g-3+r$ while they are counterexamples when $d=g-2$ and $g\geq 2r+3$ (see Theorems \ref{th:Copp1}, \ref{th:Copp2}). For chains of cycles of genus $g$ there is a necessary and sufficient condition for the existence of a divisor 
of degree $d$ and rank $r$, which is the existence of a function on a rectangle with
$g-d+r$ columns and $r+1$ rows
called $\underline{m}$-displacement tableau, where $\underline{m}$ is the vector defined by \eqref{eq:m}. For more details see \cite{CDPR}, \cite[Theorem 1]{coppens2024studyhmartenstheorem}.

We now provide an explicit example of a Martens-special chain of cycles and prove that the equality in Theorem \ref{th:Copp2} holds for $r=1$.
We first notice that, given a divisor $D$ of degree $d\leq g$ on a chain of cycles we can always characterize it as follows.

\begin{defin}
    We say that an effective divisor $D$ on a chain of cycles $\Ga$ is \emph{cycle-reduced} if $\operatorname{deg}(D|_{\gamma})\leq 1$ for any cycle $\gamma\in \Gamma$ and $\operatorname{deg}(D|_{b})=0$ for any bridge $b$. 
\end{defin}

\begin{lem}\label{lem:cyclereduced}
    Let $D$ be a divisor of degree $d\leq g$ on a chain of cycles of genus $g.$ Then there exist a cycle-reduced divisor $D'$ such that $D'\sim D.$
\end{lem}

\begin{proof}
    First, we may always assume that none of the points in the support of $D$ lies on a bridge.
    If two points are in the same cycle, one could take the linearly equivalent divisor where one of the two points is moved to the closest vertex and then again the linearly equivalent one obtained by moving such a point to the adjacent cycle through the bridge. Because of the condition on the degree, for any pairs of points in each cycle, there is always an empty one where we can move one of the two points.
\end{proof}

\begin{ex}\label{ex:Coppens}
We consider a chain of cycle of genus $5,$ with $m_3=3$ and $m_2=m_4=2$, as in Figure \ref{fg:ex_g5}.
One can check from Definition \ref{df:M_special} that it is Martens-special of rank and type $1.$ 
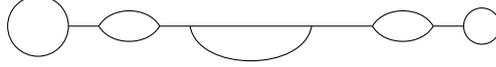
\begin{figure}[h!]
\centering
\begin{tikzcd}
\begin{tikzpicture}[scale=0.8]
    \draw (-0.5,0)circle (0.5);
    \draw (0,0)--(0.5,0);
    \draw (0.5,0) to [out=300, in=240] (1.5,0);
    \draw (0.5,0) to [out=60, in=120] (1.5,0);
    \draw (1.5,0)--(5,0);
    \draw (2,0) to [out=280, in=260] (4,0);
    \draw (5,0) to [out=300, in=240] (6,0);
    \draw (5,0) to [out=60, in=120] (6,0);
    \draw(6,0)--(6.5,0);\draw (6.8,0)circle (0.3);
    \end{tikzpicture}\end{tikzcd}\caption{A Martens-special chain of cycle of rank and type $1$ and genus $5$.}\label{fg:ex_g5}
\end{figure}

Let us show that $w^1_4=2,$ according to Theorem \ref{th:Copp2}.
By definition of Brill--Noether rank, this is equivalent to show that for any effective $E\in \operatorname{Div}^3(\Gamma)$, there exists $p\in\Gamma$ such that $E+p\in W_4^1(\Gamma).$

Let $E=\sum_{i=1}^3x_i$ for some $x_i\in \Gamma,$ and by Lemma \ref{lem:cyclereduced} we may assume that it is cycle-reduced. We have two possibilities.
\begin{enumerate}
\item If $\operatorname{deg}(E|_{C_3})=0$, we consider the point $x_i$ in the component of $\Gamma\setminus C_3$ where the degree of the divisor is 1 ($x_1$ in the Figure \ref{fg:ex_g5_p1}). Then we choose $p$ as the image of $x_i$ via the hyperelliptic involution of the component.
Then $E+p\sim 2v_3+x_j+x_k$ (or $E+p\sim 2w_3+x_j+x_k$) with $i\neq j\neq k$. 
The divisor $2v_3+x_{j}+x_{k}$ is linearly equivalent to divisors that restricted to each cycle have degree $2$. These linear equivalences are shown in Figure \ref{fg:ex_g5_p1}. Thus the rank of $E+p$ is $1.$
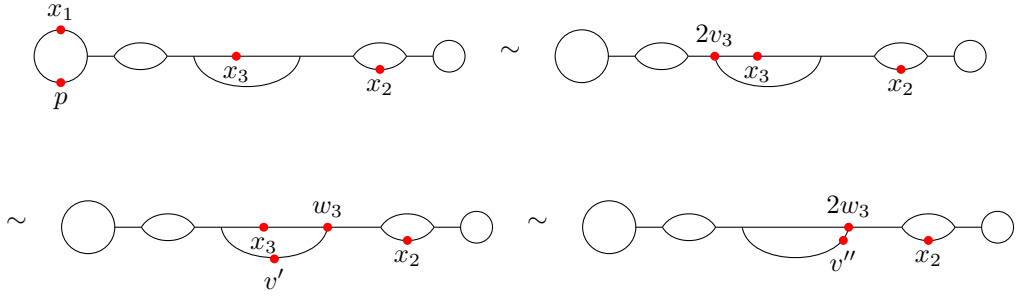
\begin{figure}[h!]
\centering
\begin{tikzcd}
\begin{tikzpicture}[scale=0.7]
    \draw (-0.5,0)circle (0.5);
    \draw (0,0)--(0.5,0);
    \draw (0.5,0) to [out=300, in=240] (1.5,0);
    \draw (0.5,0) to [out=60, in=120] (1.5,0);
    \draw (1.5,0)--(5,0);
    \draw (2,0) to [out=280, in=260] (4,0);
    \draw (5,0) to [out=300, in=240] (6,0);
    \draw (5,0) to [out=60, in=120] (6,0);
    \draw(6,0)--(6.5,0);\draw (6.8,0)circle (0.3);
    \divisor{7.1,0}
    \divisor{-0.5,0.5}\divisor{-0.5,-0.5}
    \divisor{5.5,-0.25}
     \draw (-0.5,0.5) node[anchor=south] {$x_1$};
    \draw (-0.5,-0.5) node[anchor=north] {$p$};
    \draw (5.5,-0.25) node[anchor=north] {$x_2$};
    \draw (7.1,0) node[anchor=west] {$x_3$};
    \end{tikzpicture} \sim \quad
    \begin{tikzpicture}[scale=0.7]
    \draw (-0.5,0)circle (0.5);
    \draw (0,0)--(0.5,0);
    \draw (0.5,0) to [out=300, in=240] (1.5,0);
    \draw (0.5,0) to [out=60, in=120] (1.5,0);
    \draw (1.5,0)--(5,0);
    \draw (2,0) to [out=280, in=260] (4,0);
    \draw (5,0) to [out=300, in=240] (6,0);
    \draw (5,0) to [out=60, in=120] (6,0);
    \draw(6,0)--(6.5,0);\draw (6.8,0)circle (0.3);
    \divisor{7.1,0}
    \divisor{2,0}
    \divisor{5.5,-0.25}
     \draw (2,0) node[anchor=south] {$2v_3$};
    \draw (5.5,-0.25) node[anchor=north] {$x_2$};
    \draw (7.1,0) node[anchor=west] {$x_3$};
    \end{tikzpicture}
    \\
    \sim\quad
    \begin{tikzpicture}[scale=0.7]
    \draw (-0.5,0)circle (0.5);
    \draw (0,0)--(0.5,0);
    \draw (0.5,0) to [out=300, in=240] (1.5,0);
    \draw (0.5,0) to [out=60, in=120] (1.5,0);
    \draw (1.5,0)--(5,0);
    \draw (2,0) to [out=280, in=260] (4,0);
    \draw (5,0) to [out=300, in=240] (6,0);
    \draw (5,0) to [out=60, in=120] (6,0);
    \draw(6,0)--(6.5,0);\draw (6.8,0)circle (0.3);
    \divisor{7.1,0}
    \divisor{3,-0.6}\divisor{4,0}
    \divisor{5.5,-0.25}
     \draw (3,-0.6) node[anchor=north] {$v'$};
     \draw (4,0) node[anchor=south] {$w_3$};
    \draw (5.5,-0.25) node[anchor=north] {$x_2$};
    \draw (7.1,0) node[anchor=west] {$x_3$};
    \end{tikzpicture}
    \sim\quad
    \begin{tikzpicture}[scale=0.7]
    \draw (-0.5,0)circle (0.5);
    \draw (0,0)--(0.5,0);
    \draw (0.5,0) to [out=300, in=240] (1.5,0);
    \draw (0.5,0) to [out=60, in=120] (1.5,0);
    \draw (1.5,0)--(5,0);
    \draw (2,0) to [out=280, in=260] (4,0);
    \draw (5,0) to [out=300, in=240] (6,0);
    \draw (5,0) to [out=60, in=120] (6,0);
    \draw(6,0)--(6.5,0);\draw (6.8,0)circle (0.3);
    \divisor{7.1,0}
    \divisor{3,-0.6}\divisor{5,0}
    \divisor{5.5,-0.25}
    \draw (3,-0.6) node[anchor=north] {$v'$};
    \draw (5.5,-0.25) node[anchor=north] {$x_2$};
    \draw (7.1,0) node[anchor=west] {$x_3$};
    \end{tikzpicture}
    \\
    \sim\quad
    \begin{tikzpicture}[scale=0.7]
    \draw (-0.5,0)circle (0.5);
    \draw (0,0)--(0.5,0);
    \draw (0.5,0) to [out=300, in=240] (1.5,0);
    \draw (0.5,0) to [out=60, in=120] (1.5,0);
    \draw (1.5,0)--(5,0);
    \draw (2,0) to [out=280, in=260] (4,0);
    \draw (5,0) to [out=300, in=240] (6,0);
    \draw (5,0) to [out=60, in=120] (6,0);
    \draw(6,0)--(6.5,0);\draw (6.8,0)circle (0.3);
    \divisor{7.1,0}
    \divisor{3,-0.6}\divisor{6,0}
    \divisor{5.5,0.25}
    \draw (3,-0.6) node[anchor=north] {$v'$};
    \draw (7.1,0) node[anchor=west] {$x_3$};
    \end{tikzpicture}
     \sim\quad
    \begin{tikzpicture}[scale=0.7]
    \draw (-0.5,0)circle (0.5);
    \draw (0,0)--(0.5,0);
    \draw (0.5,0) to [out=300, in=240] (1.5,0);
    \draw (0.5,0) to [out=60, in=120] (1.5,0);
    \draw (1.5,0)--(5,0);
    \draw (2,0) to [out=280, in=260] (4,0);
    \draw (5,0) to [out=300, in=240] (6,0);
    \draw (5,0) to [out=60, in=120] (6,0);
    \draw(6,0)--(6.5,0);\draw (6.8,0)circle (0.3);
    \divisor{7.1,0}
    \divisor{3,-0.6}\divisor{6.5,0}
    \divisor{5.5,0.25}
    \draw (3,-0.6) node[anchor=north] {$v'$};
    \draw (7.1,0) node[anchor=west] {$x_3$};
    \end{tikzpicture}
    \end{tikzcd}\caption{The divisor $E+p$ and linearly equivalent divisors.}\label{fg:ex_g5_p1}
\end{figure}

\item If instead $\operatorname{deg}(E|_{C_3})=1$ we can choose $p$ as above, in any of the two components of $\Gamma\setminus C_3$. The rank of $E+p$ is again at least $1$, see Figure \ref{fg:ex_g5_p} for an example. 

\begin{figure}[h!]
\centering
\begin{tikzcd}
\begin{tikzpicture}[scale=0.7]
    \draw (-0.5,0)circle (0.5);
    \draw (0,0)--(0.5,0);
    \draw (0.5,0) to [out=300, in=240] (1.5,0);
    \draw (0.5,0) to [out=60, in=120] (1.5,0);
    \draw (1.5,0)--(5,0);
    \draw (2,0) to [out=280, in=260] (4,0);
    \draw (5,0) to [out=300, in=240] (6,0);
    \draw (5,0) to [out=60, in=120] (6,0);
    \draw(6,0)--(6.5,0);\draw (6.8,0)circle (0.3);
    \divisor{2.8,0}
    \divisor{-0.5,0.5}\divisor{-0.5,-0.5}
    \divisor{5.5,-0.25}
     \draw (-0.5,0.5) node[anchor=south] {$x_1$};
    \draw (-0.5,-0.5) node[anchor=north] {$p$};
    \draw (5.5,-0.25) node[anchor=north] {$x_2$};
    \draw (2.8,0) node[anchor=north] {$x_3$};
    \end{tikzpicture} \quad\sim\quad
    \begin{tikzpicture}[scale=0.7]
    \draw (-0.5,0)circle (0.5);
    \draw (0,0)--(0.5,0);
    \draw (0.5,0) to [out=300, in=240] (1.5,0);
    \draw (0.5,0) to [out=60, in=120] (1.5,0);
    \draw (1.5,0)--(5,0);
    \draw (2,0) to [out=280, in=260] (4,0);
    \draw (5,0) to [out=300, in=240] (6,0);
    \draw (5,0) to [out=60, in=120] (6,0);
    \draw(6,0)--(6.5,0);\draw (6.8,0)circle (0.3);
    \divisor{2.8,0}
    \divisor{2,0}
    \divisor{5.5,-0.25}
     \draw (2,0) node[anchor=south] {$2v_3$};
    \draw (5.5,-0.25) node[anchor=north] {$x_2$};
    \draw (2.8,0) node[anchor=north] {$x_3$};
    \end{tikzpicture}\\
    \sim\quad
    \begin{tikzpicture}[scale=0.7]
    \draw (-0.5,0)circle (0.5);
    \draw (0,0)--(0.5,0);
    \draw (0.5,0) to [out=300, in=240] (1.5,0);
    \draw (0.5,0) to [out=60, in=120] (1.5,0);
    \draw (1.5,0)--(5,0);
    \draw (2,0) to [out=280, in=260] (4,0);
    \draw (5,0) to [out=300, in=240] (6,0);
    \draw (5,0) to [out=60, in=120] (6,0);
    \draw(6,0)--(6.5,0);\draw (6.8,0)circle (0.3);
    \divisor{2.8,0}
    \divisor{4,0}
    \divisor{3,-0.6}
    \divisor{5.5,-0.25}
    \draw (3,-0.6) node[anchor=north] {$v'$};
    \draw (4,0) node[anchor=south] {$w_3$};
    \draw (5.5,-0.25) node[anchor=north] {$x_2$};
    \draw (2.8,0) node[anchor=north] {$x_3$};
    \end{tikzpicture}
\quad\sim\quad
    \begin{tikzpicture}[scale=0.7]
    \draw (-0.5,0)circle (0.5);
    \draw (0,0)--(0.5,0);
    \draw (0.5,0) to [out=300, in=240] (1.5,0);
    \draw (0.5,0) to [out=60, in=120] (1.5,0);
    \draw (1.5,0)--(5,0);
    \draw (2,0) to [out=280, in=260] (4,0);
    \draw (5,0) to [out=300, in=240] (6,0);
    \draw (5,0) to [out=60, in=120] (6,0);
    \draw(6,0)--(6.5,0);\draw (6.8,0)circle (0.3);
    \divisor{3.9,-0.25}
    \divisor{4,0}
    \divisor{5.5,-0.25}
    \draw (3.9,-0.25) node[anchor=north] {$v''$};
    \draw (4,0) node[anchor=south] {$2w_3$};
    \draw (5.5,-0.25) node[anchor=north] {$x_2$};
    \end{tikzpicture}\end{tikzcd}\caption{The divisor $E+p$ and linearly equivalent divisors.}\label{fg:ex_g5_p}
\end{figure}
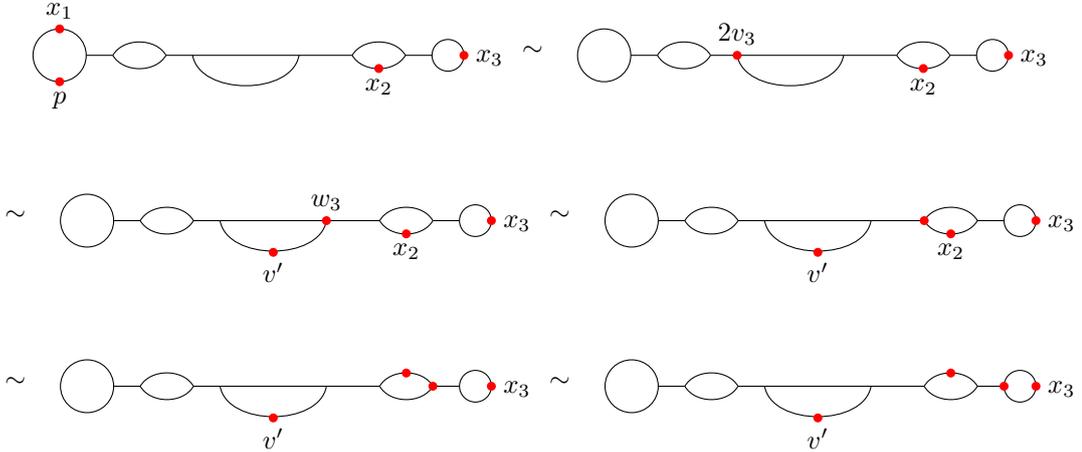
\end{enumerate}
\end{ex}

Let us stress that the argument in Example \ref{ex:Coppens} uses the fact that given any $3$ points in the graph, just by adding one point, we are able to define linearly equivalent divisors whose support has multiplicity (at least) $2$ in each cycle. 
Such a property is not natural, indeed if one slightly modifies the graph from Example \ref{ex:Coppens} the equality in Conjecture \ref{JLconj} will not hold, as in Example \ref{ex:Coppens_modified}.

\begin{ex}\label{ex:Coppens_modified}
Let us consider the graph of genus $5$ represented in Figure \ref{fg:ex_g5_mod}. It is made of cycles glued via bridges as the one in Example \ref{ex:Coppens}, but one cycle is defined by three edges instead of two.

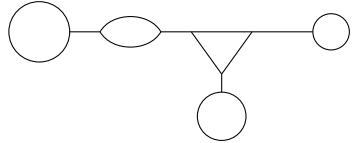
\begin{figure}[h!]
\centering
\begin{tikzcd}
\begin{tikzpicture}[scale=0.8]
    \draw (-0.5,0)circle (0.5);
    \draw (0,0)--(0.5,0);
    \draw (0.5,0) to [out=300, in=240] (1.5,0);
    \draw (0.5,0) to [out=60, in=120] (1.5,0);
    \draw (1.5,0)--(4,0);
    \draw (2,0)--(2.5,-0.7);\draw (3,0)--(2.5,-0.7);
    \draw (2.5,-1)--(2.5,-0.7);
    \draw (4.3,0)circle (0.3);
    \draw (2.5,-1.4) circle (0.4);
    \end{tikzpicture}\end{tikzcd}\caption{A metric graph of genus $5$ of cycles connected via bridges.}\label{fg:ex_g5_mod}
\end{figure}
We show that there exists an effective divisor of degree $3$ such that adding any fourth point does not yield a divisor whose support has multiplicity $2$ in any cycle, up to linear equivalence. In this case Conjecture \ref{JLconj} holds.

Let $E$ be the divisor defined by $3$ points $x_1,x_2,x_3,$ each in the interior of a distinct cycle, as in Figure \ref{fg:ex_g5_mod2}.

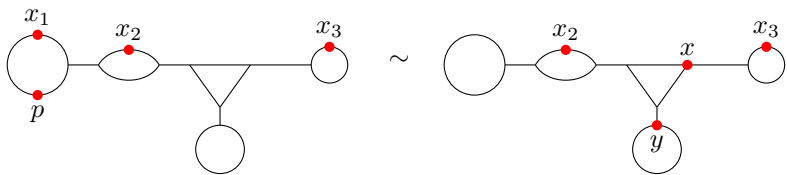
\begin{figure}[h!]
\centering
\begin{tikzcd}
    \begin{tikzpicture}[scale=0.8]
    \draw (-0.5,0)circle (0.5);
    \draw (0,0)--(0.5,0);
    \draw (0.5,0) to [out=300, in=240] (1.5,0);
    \draw (0.5,0) to [out=60, in=120] (1.5,0);
    \draw (1.5,0)--(4,0);
    \draw (2,0)--(2.5,-0.7);\draw (3,0)--(2.5,-0.7);
    \draw (2.5,-1)--(2.5,-0.7);
    \draw (4.3,0)circle (0.3);
    \draw (2.5,-1.4) circle (0.4);
    \divisor{-0.5,0.5};\divisor{1,0.25}\divisor{4.3,0.3}\divisor{-0.5,-0.5};
\draw (-0.5,0.5) node[anchor=south] {$x_1$};
\draw (-0.5,-0.5) node[anchor=north] {$p$};
    \draw (1,0.25) node[anchor=south] {$x_2$};
    \draw (4.3,0.3) node[anchor=south] {$x_3$};
    \end{tikzpicture}\quad\sim\quad
    \begin{tikzpicture}[scale=0.8]
    \draw (-0.5,0)circle (0.5);
    \draw (0,0)--(0.5,0);
    \draw (0.5,0) to [out=300, in=240] (1.5,0);
    \draw (0.5,0) to [out=60, in=120] (1.5,0);
    \draw (1.5,0)--(4,0);
    \draw (2,0)--(2.5,-0.7);\draw (3,0)--(2.5,-0.7);
    \draw (2.5,-1)--(2.5,-0.7);
    \draw (4.3,0)circle (0.3);
    \draw (2.5,-1.4) circle (0.4);
    \divisor{3,0};\divisor{1,0.25}\divisor{4.3,0.3}\divisor{2.5,-1};
\draw (3,0) node[anchor=south] {$x$};
\draw (2.5,-1) node[anchor=north] {$y$};
    \draw (1,0.25) node[anchor=south] {$x_2$};
    \draw (4.3,0.3) node[anchor=south] {$x_3$};
    \end{tikzpicture}
    \end{tikzcd}\caption{The divisor $E+p$ and a linearly equivalent one.}\label{fg:ex_g5_mod2}
\end{figure}

Let us notice that starting Dhar's burning algorithm, the whole graph burns unless we put an extra point $p,$ in the same cycle of some $x_i,$ at the same distance from the vertices within the cycle, as in Figure \ref{fg:ex_g5_mod2}. 

Then, $E+p\sim x_2+x_3+x+y$ and 
by starting Dhar's burning algorithm  from the loop containing $y$, the graph burns. 
This can be repeated for any other choice of $p$ which shows that $w_4^1(\Gamma)<2.$ 
\end{ex}

One could, however, wonder if Coppens' result generalizes to graphs obtained from Martens-special chain of cycles, by, instead of adding edges, gluing other chains of cycles along bridges.
We provide here an example of a graph $\Gamma$ obtained in this way, for which the equality in Theorem \ref{th:Copp2} applies, i.e. $w^1_{g-1}=g-3.$

\begin{ex}
Let us consider the metric graph $\Gamma$ represented in Figure \ref{fg:Hyp_tail}.
It is a Martens-special chain of cycles of rank and type $1$ and genus $7$, where we glue on one of its bridges a hyperelliptic chain of cycles, of genus $2.$ 
\begin{figure}[h!]
    \centering
    \begin{tikzcd}
    \begin{tikzpicture}[scale=0.8]
    \draw (-0.6,0)circle (0.6);
    \draw (0,0)--(0.5,0);
    \draw (0.5,0) to [out=300, in=240] (1.5,0);
    \draw (0.5,0) to [out=60, in=120] (1.5,0);
    \draw (1.5,0)--(5,0);
    \draw (6,0)--(8+2,0);
    \draw (5.5,0)circle (0.5);
    \draw (2,0) to [out=280, in=260] (4,0); 
    \draw (4.5,0)--(4.5,-1);
    \draw (4.5,-1.3)circle (0.3);
    \draw (4.5,-1.6)--(4.5,-2);
    \draw (4.5,-2.5)circle (0.5);
    \draw (5+2,0) to [out=280, in=260] (7+2,0);
    \draw (8+2,0) to [out=300, in=240] (9+2,0);
    \draw (8+2,0) to [out=60, in=120] (9+2,0);
    \draw(9+2,0)--(9.5+2,0);
    \draw (9.8+2,0)circle (0.3);
    \end{tikzpicture}\end{tikzcd}
    \caption{A metric graph $\Gamma$ obtained by gluing a Martens-special chain of cycle of rank $1$ and type $2$ and genus $7$ with a hyperelliptic chain of cycles of genus $2$ via bridges.}\label{fg:Hyp_tail}
\end{figure}
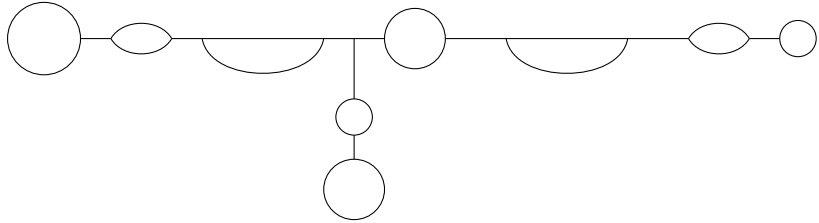

We will show that $w_{g-1}^1(\Gamma)={g-3},$
which means that any effective $E\in \operatorname{Div}^{g-2}(\Ga)$, can be completed to a divisor of rank $1$ by adding one point in its support.
For any such $E$ we consider its cycle-reduced representative, as in Figure \ref{fg:tree_cycles}.
\begin{figure}[h!]
    \centering
    \begin{tikzcd}
    \begin{tikzpicture}[scale=0.8]
    \draw (-0.6,0)circle (0.6);
    \draw (0,0)--(0.5,0);
    \draw (0.5,0) to [out=300, in=240] (1.5,0);
    \draw (0.5,0) to [out=60, in=120] (1.5,0);
    \draw (1.5,0)--(5,0);
    \draw (6,0)--(8+2,0);
    \draw (5.5,0)circle (0.5);
    \draw (2,0) to [out=280, in=260] (4,0); 
    \draw (4.5,0)--(4.5,-1);
    \draw(4.5,0) node [above] {$v$};
    \draw(5.1,-0.3) node [below] {$p$};
    \draw (4.5,-1.3)circle (0.3);
    \draw (4.5,-1.6)--(4.5,-2);
    \draw (4.5,-2.5)circle (0.5);
    \draw (5+2,0) to [out=280, in=260] (7+2,0);
    \draw (8+2,0) to [out=300, in=240] (9+2,0);
    \draw (8+2,0) to [out=60, in=120] (9+2,0);
    \draw(9+2,0)--(9.5+2,0);
    \draw (9.8+2,0)circle (0.3);
   \divisor{-1.2,0}\divisor{1,0.25}\divisor{8.5,0}\divisor{4.75,-2.1}\divisor{5.1,0.3}\divisor{5.1,-0.3}\divisor{10.5,0.25}\divisor{12.1,0}
    \end{tikzpicture}\end{tikzcd}\caption{}\label{fg:tree_cycles}
\end{figure}

Let us notice that, since $\operatorname{deg}E=g-2$, at least one of the three cycles incident to the bridges intersecting in the vertex $v$,
must contain a point in the support of $E$. Then,  we can always choose $p$ on such cycle so that $E+p\sim 2v + E'$, with $E'=E$ outside the cycle where we put $p$, and one can check that such a divisor has rank $1$. See for instance $E+p$ in Figure \ref{fg:tree_cycles}.
\end{ex}

We will generalize the result proved by Coppens for chains of cycles to metric graphs similar to the one in the above example. We will prove that for this class of graphs the Brill--Noether rank satisfies $w_{d}^r(\Ga)=d-2r$, when $d=g-2+r$, and hence they are counterexamples to Conjecture \ref{JLconj}. 

\subsection{Martens-special trees of cycles}

Let us first recall some definitions and notations from \cite[5.2]{GraphTheory} that we will use to define a new class of graphs.

A \emph{block} of a graph is a connected subgraph with no separating vertices and it is maximal with respect
to this property. 
Moreover, from \cite[5.2]{GraphTheory}, the blocks of a connected graph $G$ form a decomposition and fit together in a treelike structure. 
We associate to $G$ a bipartite graph $T_G$ with bipartition $(B,S)$, where $B$ is the set of blocks of $G$ and $S$ the set of separating vertices of $G$. A block $B$ and a separating vertex $v$ are adjacent in $T_G$ if and only if $B$ contains
$v$. 
The resulting graph $T_G$ is therefore a tree, called the \emph{block tree} of $G.$ 

With abuse of notation, we will use the same terms when referring to metric graphs. 

\begin{defin}
    A \emph{tree of cycles} is a metric graph $\Ga$ whose block-tree $T_{\Gamma}$ is such that 
    \begin{enumerate}
    \item any subgraph corresponding to a path in $T_{\Gamma}$ is a chain of cycles, up to contraction of bridges corresponding to the endpoints of the path and 
    \item any component in the block decomposition of $\Ga$ corresponding to a vertex in $T_{\Gamma}$ of valence at least $3$ is a vertex of the same valence in $\Ga$. 
    \end{enumerate}
\end{defin}

Notice that from the above definition the graph in Figure \ref{fg:ex_g5_mod} is not a tree of cycles, while the one in Figure \ref{fg:Hyp_tail} is a tree of cycles. Indeed, in Figure \ref{fg:ex_g5_mod}, the block tree has a unique vertex of valence $3$, which corresponds to the cycle defined by three edges, hence not a vertex.

\begin{rk}
    Let us observe that Lemma \ref{lem:cyclereduced} naturally extends to trees of cycles: given an effective divisor of degree equal or smaller than the genus of a metric tree of cycles we can always find a linearly equivalent divisor that is cycle-reduced.
\end{rk}

We will prove that the statement in Theorem \ref{th:Copp1} generalizes to trees of cycles. Let
us first show that any non-hyperelliptic metric graph that satisfies Conjecture \ref{JLconj} for $r=1$, also satisfies it for higher rank.
Then it is sufficient to consider $r=1.$

\begin{prop}\label{prp:induction}
Let $\Gamma$ be a non-hyperelliptic metric graph such that $w_d^1(\Gamma)<d-2.$ Then $w_d^r(\Gamma)<d-2r$ for any $r; 0<2r\leq d<g$.
\end{prop}

\begin{proof}
 By definition, $w_d^1(\Gamma)<d-2$ means that there exists an effective divisor $E$ of degree $d-1$ such that for any $p\in\Gamma,$ $\rr(E+p)=0$. 

Assume by contradiction that $w_d^r(\Gamma)=d-2r$ for some $r>1,$ i.e. any $E'$ of degree $d-r$ can be completed to a divisor of rank $r$ by adding $r$ points $p_1,\dots,p_r$. In particular we have that, for any $q_1,\dots,q_r\in\Gamma,$ 
\begin{equation}
\label{eq:ind}E'+p_1+\dots+p_r-q_1+\dots-q_r \sim E_1\geq 0.
\end{equation}

The assumption $r>1$ implies that $\operatorname{deg}E'<\operatorname{deg}E$, so we can choose $E'$ as any effective divisor contained in $E.$ For example $E=E'+x_1+\dots +x_{r-1}.$ Then, adding $x_1+\dots +x_{r-1}$ on both sides of \eqref{eq:ind} gives the linear equivalence
$$E+p_1+\dots+p_r-q_1+\dots-q_r \sim E_1+x_1+\dots+x_{r-1}.$$ 
Choose $q_j=p_j$ for $j\geq 2.$ Then 
$$E+p_1-q_1\sim E_1+x_1+\dots+x_{r-1}, $$ for any $q_1\in\Gamma$.

In other words, there exists $p_1$ such that $E+p_1\in W_d^1(\Gamma)$ contradicting the fact that $E$ cannot be completed to a divisor of rank $1.$ 
\end{proof}

\begin{prop}
Let $\Gamma$ be a tree of cycles  of genus $g$ and $d,r$ integers with $0< 2r \leq d\leq g-3+r$. Then $w^{r}_{d}(\Ga)\leq d-2r$ and equality holds if and only if $\Gamma$ is hyperelliptic.
\end{prop}

\begin{proof}
    Let $\Ga$ be a non-hyperelliptic tree of cycles. Then there exist a cycle $\gamma_0$ in $\Gamma$ of parallel edges of different lengths. 

    To show that $w^{r}_{d}(\Ga)< d-2r$ it is sufficient to provide an effective divisor $E$ of degree $d-r\leq g-3$ and show that there are no $p_1,\dots,p_r$ such that $E+\sum_{i=1}^r p_i$ has rank $r.$
    By Proposition \ref{prp:induction} it suffices to prove it for $r=1.$
    Since the degree of $E$ must be less than $g-3$, we can define its support as the sum of non-vertex points, each contained in a distinct cycle, such that the chain of (at least) 3 cycles $\Ga_0$, with $\gamma_0$ a cycle in the middle, does not contain any of such points. The case where $\Ga_0$ is a chain of exactly three cycles is represented in Figure \ref{fg:tree_cycles_prop}. 
    If $p$ is contained in $\Ga_0,$ $E+p$ cannot have rank $1.$ This can be easily seen by starting Dhar's algorithm at any other point of the same cycle containing $p.$
    If instead $p$ is contained in the same cycle with another point in the support of $E$, we can assume without loss of generality that it at the same distance from one or both vertices.
    If $\Gamma\setminus\Gamma_0$ has hyperelliptic components we have $E+p\sim D$ with $D|_{\Gamma_0}=v+p'$ where $v$ is a vertex adjacent to a bridge in $\Ga_0\setminus \gamma_0$ and $p'$ a point in the interior of the longest edge in $\gamma_0,$ as in Figure \ref{fg:tree_cycles_prop}.

    \begin{figure}[h!]
\centering
\begin{tikzcd}
\begin{tikzpicture}[scale=0.8]
    \draw (-0.6,0)circle (0.6);
    \draw (0,0)--(0.5,0);
    \draw[blue] (0.5,0) to [out=300, in=240] (1.5,0);
    \draw[blue] (0.5,0) to [out=60, in=120] (1.5,0);
    \draw[blue] (1.5,0)--(5,0);
     \draw (6,0)--(7,0);
      \draw (9-1,0)--(10-1,0);
      \draw[blue] (5.5,0)circle (0.5);
    \draw[blue] (2,0) to [out=280, in=260] (4,0); 
     \draw (8-0.5,0)circle (0.5);
    \draw(9+1,0)--(9.5+1,0);
    \draw (9.8+1,0)circle (0.3);
   \draw(5.5,-2) node [above] {$\sim$};
   \draw (8+2.5-1,0)circle (0.5);
   \divisor{8-0.5,0.5}
   \divisor{-1.2,0}
   \divisor{12.1-1,0}
    \draw(-1.2,0) node [left] {$2$};
   \divisor{10.5-1,0.5}
    \end{tikzpicture}\end{tikzcd}
    \begin{tikzcd}
\begin{tikzpicture}[scale=0.8]
    \draw (-0.6,0)circle (0.6);
    \draw (0,0)--(0.5,0);
    \draw[blue] (0.5,0) to [out=300, in=240] (1.5,0);
    \draw[blue] (0.5,0) to [out=60, in=120] (1.5,0);
    \draw[blue] (1.5,0)--(5,0);
     \draw (6,0)--(7,0);
      \draw (9-1,0)--(10-1,0);
      \draw[blue] (5.5,0)circle (0.5);
    \draw[blue] (2,0) to [out=280, in=260] (4,0); 
\draw(4.8,0) node [above] {$v$};
\draw (8+2.5-1,0)circle (0.5);
\draw (8-0.5,0)circle (0.5); 
    \draw(9+2-1,0)--(9.5+2-1,0);
    \draw (9.8+2-1,0)circle (0.3);
   \draw(3.75,-0.39) node [below] {$p'$};
   \divisor{3.7,-0.37}
   \divisor{8-0.5,0.5}\divisor{5,0}\divisor{10.5-1,0.5}\divisor{12.1-1,0}
    \end{tikzpicture}\end{tikzcd}
    \caption{Two linearly equivalent divisors on a non-hyperelliptic tree of cycles $\Gamma$, with a subgraph $\Ga_0\subset \Ga$ in blue.}\label{fg:tree_cycles_prop}
\end{figure}
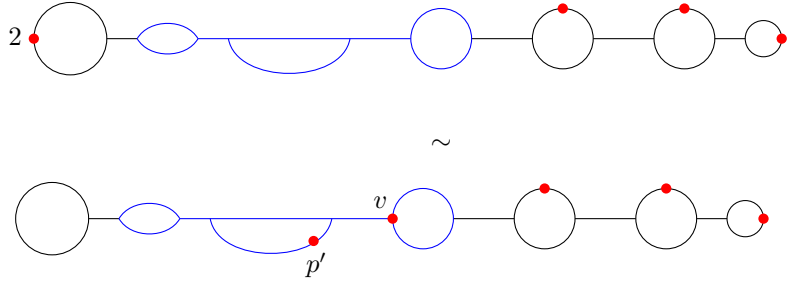
    
    Starting Dhar's burning algorithm from any other point in the cycle containing $v$ burns the whole graph, hence $E+p$ cannot have rank $r\geq 1$.

    Let us observe that if $\Gamma\setminus\Gamma_0$ does not have  hyperelliptic components, it is even more easier to obtain a divisor linearly equivalent divisor to $E+p$ whose support has a unique point which is a vertex, and all other points are non-vertex points in distinct cycles and the above argument applies. 
    
    Moreover we are considering a linearly equivalent divisor that differs from $E+p$ only on $\Gamma_0$ and the cycle containing $p.$ Therefore the same holds if the tree of cycle is not a chain.
\end{proof}

We will later see in the last section that the above result generalizes to any non-hyperelliptic graph.
We now generalize Theorem \ref{th:Copp2} to a subset of trees of cycles, defined as follows. 

\begin{defin}\label{def:MStree}
    Let $\Ga$ be a tree of cycles whose vertices corresponding to vertices of valence at least 3 in its block tree $T_{\Gamma}$ are denoted by $v_1,\dots, v_t.$ 
    A tree of cycles $\Gamma$ is called \emph{Martens-special} of rank $r$ and type $k$ if it is non-hyperelliptic and any subgraph $\Ga_i$ corresponding to a maximal path in $T_{\Gamma}$ satisfies one of the following.
    \begin{itemize}
    \item[(P1)] It is hyperelliptic.
    \item[(P2)] It is Martens-special of rank $r_i$ and type $k_i$ or 
    \item[(P3)] It is a chain of cycles that becomes Martens-special of rank $r'$ and type $k_i$ by adding at some $w\in\{v_1,\dots,v_t\}\cap \Ga_i$ a hyperelliptic chain of cycles of genus $r'$. 
    \end{itemize}
    Here,
    $r=\operatorname{min}\{r_i|\, \Ga_i \text{ Martens-special of rank $r_i$ and type $k_i$}\}$ and $
    r=1$ if the set is empty (i.e. all $\Ga_i$ satisfy either (P1) or (P3)) and
    $k$ is the number of non-hyperelliptic cycles.
\end{defin}
\begin{rk}
    In the above definition, when we consider the subgraph $\Ga_i$ which is neither hyperelliptic nor Martens-special, it is sufficient to add hyperelliptic cycles only in correspondence of vertices $w$ in the first or last bridge of the chain or on bridges between non-hyperelliptic cycles to make it Martens-special. 
    This is because, from Definition \ref{df:M_special}, a chain of cycles is Martens-special of rank $r$ if the first $r+1$ and the last $r+1$ cycles are hyperelliptic and between any two non-hyperelliptic cycles there at least $r$ consecutive hyperelliptic ones, or in other words there is a hyperelliptic chain of genus at least $r$.
\end{rk}

In the above definition, if the block tree $T_{\Ga}$ is a path we recover the definition of a Martens-special chain of cycles of rank $r$ and type $k.$ 

\begin{ex}
In Figure \ref{fg:ex_semispecial} we show an example of a Martens-special tree of cycles, where none of the chains corresponding to a maximal path in the block tree is Martens-special, but they all satisfy (P3).
For instance, consider the horizontal chain of cycles. It becomes Martens-special only after adding a hyperelliptic cycle at $v$, as shown in Figure \ref{fg:max_chain}. The same holds if we consider the other subgraphs corresponding to the other two maximal paths in the block tree.

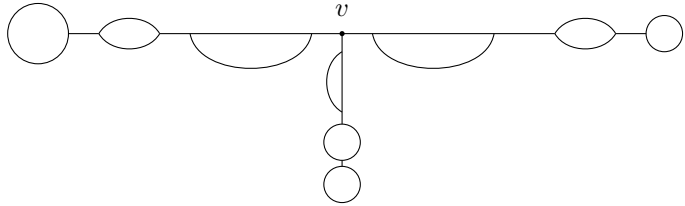
\begin{figure}[h!]
\centering
\begin{tikzcd}
\begin{tikzpicture}[scale=0.8]
    \draw (-0.5,0)circle (0.5);
    \draw (0,0)--(0.5,0);
    \draw (0.5,0) to [out=300, in=240] (1.5,0);
    \draw (0.5,0) to [out=60, in=120] (1.5,0);
    \draw (1.5,0)--(8,0);
    \draw (2,0) to [out=280, in=260] (4,0);
 \draw (5,0) to [out=280, in=260] (7,0);
    \draw (7+1,0) to [out=300, in=240] (8+1,0);
    \draw (7+1,0) to [out=60, in=120] (8+1,0);
    \draw(8+1,0)--(9.5,0);\draw (9.8,0)circle (0.3);
    \vertex{4.5,0}
 \draw (4.5,0)--(4.5,-1.5);
     \draw (4.5,-0.3) to [out=300-90, in=240-90] (4.5,-1.3);
     \draw (4.5,-1.8)circle (0.3);
     \draw (4.5,-2.1)--(4.5,-2.2);
    \draw (4.5,-2.5)circle (0.3);
    \draw (4.5,0.3) node {$v$};
    \end{tikzpicture}\end{tikzcd}\caption{A Martens-special tree of cycles of rank $1$, type $3$ and genus $9$.}\label{fg:ex_semispecial}
\end{figure}

\begin{figure}[h!]
\centering
\begin{tikzcd}
\begin{tikzpicture}[scale=0.8]
   \draw (-0.5,0)circle (0.5);
    \draw (0,0)--(0.5,0);
    \draw (0.5,0) to [out=300, in=240] (1.5,0);
    \draw (0.5,0) to [out=60, in=120] (1.5,0);
    \draw (1.5,0)--(8,0);
    \draw (2,0) to [out=280, in=260] (4,0);
 \draw (5,0) to [out=280, in=260] (7,0);
    \draw (7+1,0) to [out=300, in=240] (8+1,0);
    \draw (7+1,0) to [out=60, in=120] (8+1,0);
    \draw(8+1,0)--(9.5,0);\draw (9.8,0)circle (0.3);
    \vertex{4.5,0}
      \draw (4.5,0.3) node {$v$};
\end{tikzpicture}\\
\begin{tikzpicture}[scale=0.8]
    \draw (-0.6+1,0)circle (0.6);
    \draw (1,0)--(0.5+1,0);
    \draw (1.5,0) to [out=300, in=240] (2.5,0);
    \draw (1.5,0) to [out=60, in=120] (2.5,0);
    \draw (2.5,0)--(5.5,0);
     \draw (6.5,0)--(8+2,0);
      \draw (6,0)circle (0.5);
    \draw (3,0) to [out=280, in=260] (5,0); 
        \draw (5+2,0) to [out=280, in=260] (7+2,0);
    \draw (8+2,0) to [out=300, in=240] (9+2,0);
    \draw (8+2,0) to [out=60, in=120] (9+2,0);
    \draw(9+2,0)--(9.5+2,0);
    \draw (9.8+2,0)circle (0.3);
\end{tikzpicture}\end{tikzcd}\caption{A chain of cycles which becomes Martens-special of rank $1$ and type $2$ after replacing the vertex $v$ with a hyperelliptic cycle.}
    \label{fg:max_chain}
\end{figure}
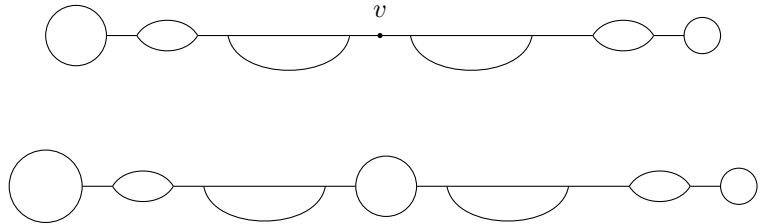
\end{ex}

In the rest of the section we prove that Martens-special trees of cycles behave like Martens-special chains of cycles, namely they realize the equality in Proposition \ref{prp:ineq}.

\begin{prop}\label{prp:main_tree_cycles}
    Let $\Gamma$ be a Martens--special tree of cycles of genus $g$ and rank $r\leq g-2$.
    Then $w_{g-2+r}^r(\Gamma)=g-2-r.$ 
\end{prop}
\begin{proof}
    Let us assume that $\Gamma$ has a block-tree $T_{\Gamma}$ with precisely a vertex $v$ of valence $3$: the complement of a chain of cycles corresponding to a maximal path is again a chain of cycles.
    The proof for the general case can be obtained inductively by gluing chains of cycles to $\Gamma$.

    Let $E$ be any cycle-reduced divisor of degree $g-2$. We need to show that there exist $p_1,\dots, p_r$ such that $E+\sum_{i=1}^r p_r$ has rank $r$.
    
    We consider a maximal chain $\Gamma_1$ such that $\operatorname{deg}(E_1)=g(\Ga_1)-2,$ where $E_1:=E|_{\Gamma_1}.$ 
    Let us assume that $\Ga_1$ satisfies property (P3) in Definition \ref{def:MStree}, then we replace the vertex $v$ with a hyperelliptic chain $H$ of genus $r$ and obtain a Martens-special chain $\Ga_1'$ of rank $r$. This is clear in the case $r=1.$ If $r>1$ this is also true because by definition there is at least a Martens-special chain of cycles contained in $\Ga$ of rank (at least) $r$ and the two subgraphs in $\Ga_1$ obtained by removing $v$ must be contained in Martens-special chains of rank higher or equal than $r$. 
    
    Then by Theorem \ref{th:Copp2} if we consider the effective divisor $E_1+q_1+\dots+q_r$ on $\Ga_1'$, with $q_i$ in the interior of distinct edges, each in a distinct cycle of $H$, then there exist $q'_i$ such that $\operatorname{r}(E_1+\sum_{i=1}^r(q_i+q_i'))\geq r$ on $\Ga_1'.$ 
    
    If $q_i'\not \in H,$ for any $i$, then we take $p_i=q_i'$ and $\operatorname{r}(E_1+\sum_{i=1}^r p_i)\geq r$ on $\Ga_1.$ Therefore $E_1+\sum_{i=1}^r p_i\sim r\cdot v+E_1'$ with $\operatorname{Supp}E_1'\subset \Ga_1$. Thus $E+\sum_{i=1}^r p_i\sim r\cdot v+ E|_{\Ga\setminus \Ga_1}+E_1'$ which restricted to $\overline{\Ga\setminus \Ga_1}$ has also rank (at least) $r$ by Riemann--Roch.

    If $q_i'\in H,$ for some $i\in\{1,\dots,r\}$, then for each of such $i$ we have two possibilities: either $E_1+q_i+q_i'\sim E_1+2w,$ with $w$ one of the two points in the boundary $\partial H$ in $\Ga_1'$, or $E_1+q_i+q_i'\sim E_1+w+w',$ with $w'\in H,$ not a vertex (up to relabeling the points). In the latter case, we just repeat the above argument and then $w\notin H$ behaves as $q_i'\not \in H$. 

    Otherwise we may assume $E_1+\sum_{i=1}^r (q_i+q_i')\sim E_1+(2s)w+ \sum_{j=s+1}^r (q_j+q_j'),$ for the first $s$ pairs of points, with $s;\,1\leq s\leq r.$
    Then we have 
    \begin{equation}\label{eq:prp1}
    \rr(E_1+\sum_{i=1}^r (q_i+q_i'))=\rr(E_1+(2s)w+ \sum_{j=s+1}^r (q_j+q_j'))\geq r\qquad\text{ on }\Ga'_1.
    \end{equation}

Notice that since $H$ is hyperelliptic and $q_j'\notin H$ for $j=s+1,\dots,r,$ then  
    \begin{equation}\label{eq:prp2}
    \rr(E_1+(2s)v+ \sum_{j=s+1}^r q_j')\geq r\qquad\text{ on }\Ga_1.
    \end{equation}

    Moreover, removing $s$ points cannot decrease the rank more than $s$, then from \eqref{eq:prp1} we have
    \begin{equation}\label{eq:prp3}
    \rr(E_1+\sum_{i=1}^r q_i+\sum_{j=s+1}^r q_j')\geq r-s\qquad\text{ on }\Ga'_1.
    \end{equation}

    Again, if we restict to $\Gamma_1,$ then 
    \begin{equation}\label{eq:prp4}
    \rr(E_1+\sum_{j=s+1}^r q_j')\geq r-s\qquad\text{ on }\Ga_1.
    \end{equation}
    
    Then we take $p_j=q_j'$ for $j=s+1,\dots,r$ and for $i=1,\dots, s,$ let $p_i$ be a point in the $i$-th closest cycle in ${\Ga\setminus \Ga_1}$ to $v$ at the same distance from $v$ as the point already in the support of $E$. 
    Notice that we have already observed that when $r>1$ the subgraph in ${\Ga\setminus \Ga_1}$ incident $v$ has to contain (at least) $r\geq s$ consecutive hyperelliptic cycles.
    Then 
    \begin{equation}\label{eq:prp6}
    E+\sum_{i=1}^r p_i \sim E_1+ (2s)v +\sum_{j=s+1}^r p_j+R,
    \end{equation}
    for some effective divisor $R$ with $\operatorname{Supp}(R)\subset \Ga\setminus \Ga_1$ and degree $g(\Ga\setminus \Ga_1)-s$. 
    Such a divisor has rank (at least) $r$ on $\Ga_1$ from \eqref{eq:prp2}. 
    Moreover from \eqref{eq:prp4}, we can write $E_1+\sum_{j=s+1}^r p_j\sim (r-s)v+R',$ with $\operatorname{Supp}(R')\subset\Ga_1.$ Then \eqref{eq:prp6} rewrites as
    \begin{equation}\label{eq:Ga_12}
    E+\sum_{i=1}^r p_i \sim E_1+ (2s)v +\sum_{j=s+1}^r p_j+R\sim (r+s)v+R'+R.
    \end{equation}
    Its restriction on $\overline{\Ga\setminus\Ga_1}$ is $(r+s)v+R$, which has rank (at least) $r$ by Riemann-Roch.

    If $\Ga_1$ is hyperelliptic or Martens-special of rank $r$ we just consider $\Ga_1'=\Ga_1$ without adding the points $q_i$ and the claim follows more easily.
    Repeating the argument for any $E$ concludes the proof. 
\end{proof}

\begin{theo}\label{th:main_tree_cycles}
    Let $\Gamma$ be a tree of cycles of genus $g$.
    Then $w_{g-2+r}^r(\Gamma)=g-2-r$ for some integer $r\in\{1,\dots, g-2\}$ if and only if $\Gamma$ is hyperelliptic or it is a Martens-special tree of cycles of rank at most $r$.
\end{theo}

\begin{proof}
    From Propositions \ref{prp:ineq}, \ref{prp:main_tree_cycles}, it is enough to show that if $\Ga$ is non-hyperelliptic such that $w_{g-2+r}^r(\Gamma)=g-2-r$, then it is Martens-special of rank at most $r$.
    As in the previous proof, we assume that $\Gamma$ has a block-tree $T_{\Gamma}$ with precisely a vertex $v$ of valence $3$.

    Let $\Gamma_1$ be one of the chains corresponding to a maximal path in $T_{\Ga}$ and
    consider $E_1$ an effective divisor on $\Ga_1$ of degree $g_1-2$, $g_1=g(\Ga_1)$. Then adding any $g-g_1$ points yields a divisor of degree $g-2$ over $\Ga.$ We choose the points $\{q_j\}_{j=1,\dots, g-g_1}$ in $\Gamma\setminus \Ga_1,$ each in a distinct cycle, in the interior of the edges. 
    Then, from the equality $w_{g-2+r}^r(\Ga)=g-2-r$, there exist $p_1,\dots,p_r$ such that 
    \begin{equation}\label{eq:hp}
        \rr(E_1+\sum_{j=1}^{g-g_1} q_j+\sum_{i=1}^r p_i)\geq r.
    \end{equation}
    Notice that among the points $p_1,\dots,p_r$, some might be contained in $\Ga\setminus \Ga_1.$ For instance, if $p_1\in \Ga\setminus \Ga_1$, there exists $\hat j\in\{1,\dots,g-g_{1}\}$ such that $E_1+\sum_{j=1}^{g-g_1} q_j+p_1\sim E_1+\sum_{j\neq \hat j} q_j+2v$ (otherwise we would have $E_1+\sum_{j=1}^{g-g_1} q_j+p_1\sim E_1+\sum_{j=1}^{g-g_1} q'_j+v$ with the points $q'_j$ with the same properties as $q_j$ so in that case we can assume $p_1=v$ which thus belong to $\Ga_1$).
    Therefore on $\Gamma_1$
    \begin{equation}\label{eq:thm1}
        \rr(E_1+p_1)\geq 1 \quad\text{or}\quad
        \rr(E_1+2v)\geq 1.
    \end{equation}
    
    Let us consider now a chain of cycles $\Ga_1',$ obtained by replacing $v$ with a hyperelliptic chain of cycles $H$ of genus $r$. Notice also that adding $H$ to a hyperelliptic or Martens-special chain of cycles yields again a hyperelliptic or Martens-special chain (of higher or equal rank), respectively. 
    
    We will prove that $w_{g(\Ga_1')-2+r}^r(\Ga_1')=g(\Ga_1')-2-r$ (by construction $g(\Ga_1')=g_1+r\geq r+2$).
    By Theorem \ref{th:Copp2}, the equality would imply that $\Ga_1'$ is hyperelliptic or Martens-special of rank $r$, which means that the chain $\Ga_1$ is hyperelliptic, Martens-special of rank at most $r,$ or a chain of cycles satisfying (P3) in Definition \ref{def:MStree}.
    For any $E_1'$, a cycle-reduced divisor of degree $g_1+r-2$ on $\Ga_1',$ we want to show that there exist $r$ points that, added to $E_1'$, define a divisor of rank at least $r$.
    
    Denote by $w_1,w_2$ the two points in $\partial H$, whose image under the contraction of $H$ is $v.$ 
    Since $E_1'$ in cycle-reduced, we can always write $E_1'=E_1+x_1+\dots +x_r,$ with $\operatorname{Supp}(E_1)\subset \Ga_1$, with $\operatorname{deg}(E_1)=g_1-2$. Then, depending on the position of the points $x_j$ with respect to $H,$ we have two possibilities.

    \begin{enumerate}
        \item If $x_j\in H$ for all $j=1,\dots,r$, then 
    let us treat first the case $r=1$ separately.
    
    From \eqref{eq:thm1}, we have either $\rr((E_1+p)|_{\Ga_1})\geq 1$ for some $p\in\Ga_1$ or $\rr((E_1+2v)|_{\Ga_1})\geq 1$.
    In the first case $E_1+p\sim w_i+R,$ for some effective divisor $R,$ and $\rr(w_i+x_1)=1$ on $H$. 
    In the second one instead we choose $p=\iota(x_1)$, where $\iota$ is the hyperelliptic involution in $H.$ Then  $p+x_1\sim 2w_i,$ which has rank $1$ in $H$ and yields that $E_1+2v$ has also rank $1$, when restricted to $\Ga_1.$
    Therefore in both cases we obtain a divisor of rank at least $1$ on $\Ga_1'$ by adding a point.
    
    For $r\geq 2,$ we can choose $p_j=\iota (x_j),$ for any $j=1,\dots,r.$
    Then, $E_1+\sum_{j=1}^r(x_j+p_j)\sim E_1+(2r)w_i$ which clearly has rank $r$ on $H$.
    By applying the Riemann--Roch theorem on $\Ga_1$ we have  $\operatorname{r}(E_1+(2r)w_i)\geq 2r-2\geq r $.

    \item If there exist some $j$ such that $x_j\notin H$, 
    let us consider again the case $r=1$ first.
    
    In particular $\overline{\Ga_1'\setminus H}$ is defined by two chains of cycles $\Ga^a$ and $\Ga^b,$ such that the degree of $E_1+x_1$ restricted to one of the two components, is equal to the genus.
    Say that such a component is $\Ga^a$ and $w_1\in\Ga^a$ and $w_2\in \Ga^b.$ 
    Pick $p_1\in \Ga^a$ at the same distance from $H$ and on the opposite edge (if not on a bridge) as the closest point (w.r.t. $H$), $z_1$, in the support of $(E_1+x_1)|_{\Ga^a}.$ This gives $(E_1+x_1+p_1)\sim 2w_1+(E_1+x_1-z_1)\sim 2w_2+(E_1+x_1-z_1)$, which has rank $1$ on $H.$ Riemann--Roch applied on both components $\Ga^a,\Ga^b$ implies that the rank has to be at least 1.
    
    When $r\geq2,$ since the total degree of $E_1'$ is $g_1+r-2$ and $E'_1$ is cycle reduced, there are at most two points among the $x_j$ not in $H$. Say $x_1,x_2\notin H$ (the case with one point is analogous) and we may now assume that the degree of $E_1+\sum_{j=1}^r x_j$, restricted to each of the two components, is equal to their genera. Define $p_1, z_1\in \Ga^a$ as above, and $p_2,z_2$ similarly on $\Ga^b.$ This gives $E_1+x_1+x_2+p_1+p_2\sim E_1-z_1-z_2+x_1+x_2+4w_i.$
    Define instead $p_j=\iota(x_j)$ for $j\in\{ 3,\dots, r\}.$ Then $E_1+\sum_{j=1}^r (x_j+p_j)\sim (2r)w_i+R$ for some effective divisor $R$ of degree $g_1-2,$ with support in $\Ga_1.$
    Again, by Riemann--Roch, $\operatorname{r}((2r)w_i+R)|_{\Ga_1})\geq 2r-2\geq r$ when $r\geq2.$ 
    \end{enumerate}
\end{proof}

\subsection{Further counterexamples}\label{ssc:further}

Unfortunately, Martens-special trees of cycles (and thus chains of cycles) are not the only counterexample when $d=g-2+r.$
Here we present another class of non-hyperelliptic graphs for which Conjecture \ref{JLconj} does not hold.
Namely, $w_{g-2+r}^r(\Ga)=g-2-r$ when $r=1$ and $\Ga$ is described as follows.
We will consider a metric graph $\Ga$ of genus $g\geq 7$ with two parallel edges of different lengths with vertices $v_1,v_2$ and attached to $v_i$ is glued a bridge with endpoint $w_i$, $i=1,2$; moreover on each $w_i$ is glued a hyperelliptic graph $\Ga_i$ of genus $g_i\geq 3$, with $2w_i\in W_2^1(\Ga_i)$, as in Figure \ref{fg:controesnuovo}. 

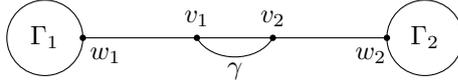
\begin{figure}[h!]
\centering
\begin{tikzcd}
\begin{tikzpicture}
   \draw (-1.5,0)circle (0.5);
   \draw (3.5,0)circle (0.5);
     \draw (-1,0)--(3,0);
    \draw (0.5,0) to [out=300, in=240] (1.5,0);
    \draw (0.5,0.2) node {$v_1$};
    \draw (-1+0.3,0) node[below] {$w_1$};
    \draw (2.29+0.5,0) node[below] {$w_2$};
    \draw (1,-0.2) node[below] {$\gamma$};
    \draw (1.5,0.2) node {$v_2$};
     \draw (-1.5,-0.09) node {$\Gamma_1$};
      \draw (3.5,-0.09) node {$\Gamma_2$};
    \vertex{0.5,0}
\vertex{3,0}
\vertex{1.5,0}
\vertex{-1,0}
\end{tikzpicture}\end{tikzcd}\caption{  
  A non-hyperelliptic metric graph with a cycle consisting of edges of different lengths, glued via bridges to two hyperelliptic subgraphs $\Ga_1$ and $\Ga_2$.}
    \label{fg:controesnuovo}
\end{figure}

\begin{prop}\label{prp:new_counterex}
    Let $\Ga$ be a graph of genus $g\geq 7$, as in Figure \ref{fg:controesnuovo}. Then $w_{g-1}^1(\Gamma)=g-3.$
\end{prop}

\begin{proof}
Let $E$ be any divisor of degree $g-2.$ Let us notice that, similarly to the chains of cycles, we can always assume, up to linear equivalence, that $E$ has degree at most one in the cycle defined by the two parallel edges of different lengths, denoted by $\gamma$.
We also denote by $g_i$ the genus of the hyperelliptic subgraph $\Ga_i.$ Then $g=g_1+g_2+1.$

We consider all possibilities for $E$ and show that, for each of them, we can always find a $p\in\Ga$ such that $\operatorname{r}(E+p)\geq 1.$
\begin{enumerate}
\item Let us consider first the case where $\operatorname{deg}E|_{\gamma}=1$.
Then, there exists some $i\in\{1,2\}$ such that $\operatorname{deg}E|_{\Gamma_i}\neq 0.$ We place $p\in\Ga_i,$ as the image of any point in the support of $E|_{\Gamma_i}$, under the hyperelliptic involution. Then  $E|_{\Ga_i}+p\sim 2v_i.$ In particular, $E+p$ is linearly equivalent to a divisor of degree $3$ when restricted to $\gamma$ which yields a linearly equivalent divisor with $2w_j$, with $j\neq i$, in its support, with rank $1$ over $\Ga_j$ since $w_j\in W^1_2(\Ga_j)$. 

    \item Now we consider the case where $\operatorname{deg}E|_{\gamma}=0$.
    \begin{enumerate}
        \item If $\operatorname{deg}E|_{\Gamma_1}=0,$ then  $\operatorname{deg}E|_{\Gamma_2}=g-2=g_1+g_2-1.$   Any $p\in\Ga_2$ then yields a divisor $E+p$ of degree $g_1+g_2$ and by Riemann--Roch such a divisor, over $\Ga_2$ has also rank $\operatorname{r}(E+p)\geq g_1\geq3$. Therefore we can consider an equivalent divisor that contains $3w_2$ in its support, which in turn is also linearly equivalent to a divisor with $2w_1$ in its support, which has rank $1$ on $\Ga_1.$

    \item If $\operatorname{deg}E|_{\Gamma_i}=k_i> 0,$ for each $i,$ the we have $k_1+k_2=g_1+g_2-1.$ 
    In particular if $k_1\leq g_1-1$ then $k_2\geq g_2.$
    In this case we put $p\in\Gamma_1$ as the image of any point in the support of $E|_{\Gamma_1}$, via the hyperelliptic involution. Therefore, $E+p$ has rank $1$ on $\Gamma_1$ and it is linearly equivalent to a divisor with $2v_1\sim w_2+p'$ in its support, where $p'$ is a point in the interior of $\gamma.$ Such a linearly equivalent divisor has rank $1$ also in $\Gamma_2:$ it is a divisor of degree $k_2+1\geq g_2+1$ and the claim follows from Riemann--Roch.

    If instead $k_1\geq g_1$ then $k_2\leq g_2-1$ and we repeat the above with $p\in\Gamma_2,$ with the same properties.
    \end{enumerate}
\end{enumerate}
\end{proof}

Similarly to Martens-special trees of of cycles, if we take instead $d\leq g-3+r$ then Conjecture \ref{JLconj} holds, namely $w_{d}^r(\Ga)<d-2r.$ We provide here a proof for $r=1$ which suffices to prove the result for $r>1$ by Proposition \ref{prp:induction}.

\begin{prop}\label{prp:new_counterex2}
     Let $\Ga$ be as in Proposition \ref{prp:new_counterex} and $d\leq g-2$. Then $w_{d}^1(\Ga)<d-2.$
\end{prop}
\begin{proof}
    If suffices to explicitly construct an effective divisor $E$ of degree $d-1$ and show that there exists no $p\in\Ga$ such that $E+p$ has rank $1.$

    Since $2w_i\in W_2^1(\Ga_i),$ there is precisely one cycle in (each component of) $\Ga_i$ containing $w_i.$  We define $E$ by considering $d_i\leq g(\Ga_i)-1$ points on each $\Ga_i,$ 
    each in distinct cycles and that no two points are exchanged via the hyperelliptic involution.  We further require that the cycles containing $w_i$ do not contain such points. It is always possible to choose such a divisor (this will be later proved in Lemma \ref{lm:cycles}).
    Notice that on $\ga$ we have $\deg (E|_{\gamma})=0$.

    Suppose now by contradiction that there exists $p$ such that $\operatorname{r}(E+p)\geq 1,$ then $p\in \Ga_i$ for some $i,$ and it is the image under the corresponding hyperelliptic involution of some point $x\in \operatorname{Supp}(E|_{\Ga_i})$, \cite[Theorem 3.6]{L}. Then we have $E+p\sim w_j+p'+(E|_{\Ga_i}-x)+(E|_{\Ga_j})$ with $j\neq i$ and $p'$ a point in the interior of an edge in $\gamma$.
    Starting a fire in the cycle containing $w_j,$ away from $w_j$, it burns the whole graph, giving a contradiction.
\end{proof}

The graph considered above, together with Martens-special trees (and chains) of cycles are not the only counterexamples to Conjecture \ref{JLconj}, consider for instance a hyperelliptic graph (of sufficiently large genus) where we change the length of exactly one of the edges not contracted via the degree $2$ map to a tree.

We will prove in the next section that such graphs are no longer a counterexample for the conjecture if $d\leq g-3+r$, as we have seen in Proposition \ref{prp:new_counterex2}.

\section{Jensen-Len's conjecture for $d\leq g-3+r$}\label{sc:main_proof}

Let us recall that all the graphs discussed in the previous section provide counterexamples to Conjecture \ref{JLconj} only when $d=g-2+r.$
In particular, we have proved that for non-hyperelliptic trees of cycles and graphs in Subsection \ref{ssc:further} the equality for the Brill--Noether rank cannot be achieved if $d\leq g-3+r$. 

This section is devoted to prove that the Conjecture \ref{JLconj} is indeed true for $d\leq g-3+r$ and thus prove Theorem \ref{th:main_theo}. 

We will prove this first for $r=1,$ in Proposition \ref{prp:no_sep} and the result generalizes to $r>1$ by Proposition \ref{prp:induction}.
In order to prove \cref{prp:no_sep} we will construct a divisor $E=x_1+\dots+x_{d-1}$; $d < g-1$ satisfying the following two properties.
\begin{itemize}
    \item[(E1)] All $x_i$ are such that there exists a cycle containing them uniquely.
\end{itemize}

This in particular yields a divisor of trivial rank. Moreover such a divisor always exists because of the following. 

\begin{lem}\label{lm:cycles}    Let $\Gamma$ be a metric graph of genus $g$. Then there exist $g$ distinct edges $e_1,\dots,e_{g}$ such that for each $i$ there exists a cycle in $\Gamma$ that contains $e_i$ and not $e_j$ for any $j\neq i.$\end{lem}\begin{proof}It is sufficient to choose a spanning tree of the graph. Then the complement is given by $g$ edges and we can take $e_1,\dots, e_g$ to be such edges.\end{proof}

Therefore, one can place the points $x_i$ in the interior of some of the $e_j$, as in \cref{lm:cycles}.
We further require that the points $x_i$ are not placed at the same distances with respect to the endpoints of the edges containing them.

\begin{itemize}
    \item[(E2)] Denote by $e_i$ the edge containing $x_i$ in its interior and $u_i^{(1)}, u_i^{(2)}$ the endpoints of $e_i.$ We require $d(x_i,u_i^{(j)})$ to be all different from each other, and $d(x_i,u_i^{(j)})\neq \sum_{e\in S\subseteq E(G_0)} l(e)$ for $i=1,\dots,d-1$, $j=1,2.$
\end{itemize}
As (E1), this property can also be satisfied since we are considering finite graphs.

Finally, in order to obtain a proof of Theorem \ref{th:main_theo} it is sufficient to prove the following.
\begin{prop}\label{prp:no_sep}
    Let $\Gamma$ be non-hyperelliptic of genus $g$. Let $d\leq g-2,$ then $w^1_d(\Ga)<d-2.$
\end{prop}

\begin{proof}
    For simplicity we assume our graph to be $2$-edge connected since the contraction of bridges does not change the rank of a divisor. 
    Similarly to the proof of Proposition \ref{prp:new_counterex2}, we construct a divisor $E$ of degree $d-1$ such that no other point $p$ yields $\rr(E+p)\geq1.$

    Since $\Ga$ is non-hyperelliptic, from \cite{MC}, then $\Gamma$ has at least three vertices (or at least two if it contains loops) and  there exist a non-loop cycle $\gamma$, whose edges satisfy one the following.
    \begin{itemize}
    \item[(NH1)] Two edges of $\gamma$ define a $2$-edge cut of $\Gamma$ such that:
    \begin{itemize}
        \item[a.] the $2$-edge cut is defined by edges of different length; or
        \item[b.] the $2$-edge cut is defined by non-parallel edges and there exists another cycle in $\Gamma$ containing only one of the two distinct endpoints of such an edge-cut.
    \end{itemize}
    \item[(NH2)] No pair of edges in $\gamma$ defines a $2$-edge cut of $\Ga$. 
    \end{itemize}

    Notice that, in the description of $\gamma$ in (NH1)a, we also consider for instance a hyperelliptic graph, where we change the length of an edge not fixed by the hyperelliptic involution, as mentioned at the end of the last section, see for instance Figure \ref{fg:controesnuovo}. An example of a graph whose cycle $\gamma$ instead satisfies (NH1)b is shown in Figure \ref{fg:quasihypgraph}.

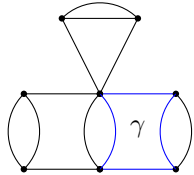
\begin{figure}[h!]
\centering
\begin{tikzcd}
\begin{tikzpicture}
     \draw (0,0)--(1,0);
 \draw (0,1)--(1,1);
    \vertex{0,0}
 \vertex{1,0}
  \vertex{0,1}
   \vertex{1,1}
\draw (0,1) to [out=-45, in=45] (0,0);
\draw (0,1) to [out=-45-90, in=45+90] (0,0);

    \draw[blue] (0+1,0)--(1+1,0);
 \draw[blue] (0+1,1)--(1+1,1);
    \vertex{0+1,0}
 \vertex{1+1,0}
  \vertex{0+1,1}
   \vertex{1+1,1}
\draw[](0+2,1) to [out=-45, in=45] (0+2,0);
\draw[blue] (0+2,1) to [out=-45-90, in=45+90] (0+2,0);
\draw(1.5,0.5) node{$\gamma$};
\draw (1,1)[blue] to [out=-45, in=45] (1,0);
\draw (1,1) to [out=-45-90, in=45+90] (1,0);
  \draw (1,1)--(0.5,2);
 \draw (1,1)--(1.5,2);
 \draw (0.5,2)--(1.5,2);

 \draw (0.5,2) to [out=45, in=45+90] (1.5,2);
  \vertex{0.5,2}
  \vertex{1.5,2}
\end{tikzpicture}\end{tikzcd}\caption{  
  A non-hyperelliptic metric graph with a cycle $\gamma$ satisfying (NH1)b.}
    \label{fg:quasihypgraph}
\end{figure}
    If $\gamma$ satisfies (NH1)a, we denote $\gamma',\gamma''$ the two cycles incident to $\gamma$ (i.e. sharing some of its points), in distinct components of $\Ga$ with the $2$-edge cut removed, such that both contain the vertex of the shortest edge in such a cut.
    If $\gamma$ satisfies (NH1)b, let $\gamma'$ be the cycle containing only one of two distinct endpoints of the cut.

    Instead, if $\gamma$ satisfies (NH2), let $k_0\geq 3$ be then the minimum for which a pair of edges in $\gamma$ determines a non-trivial edge cut.
    We choose $\gamma'$ to be any cycle containing any of the two edges $e\in\{e_1,e_2\}$ in $\gamma$ defining such a cut. Notice that this is always possible otherwise the edges in $\gamma$ would determine a $2$-edge cut and we would be in the previous case.
    We also define $\gamma''$ as a cycle, distinct to $\gamma,\gamma'$ containing an edge consecutive to $e$ in $\gamma.$  

    Let $E:=\sum_{i=1}^{d-1} x_i$ with $x_1,\dots,x_{d-1}$ satisfying (E1), (E2), where we further require them not be in $\gamma, \gamma'$ and $\gamma''$, when the latter is defined. 
    This is always possible as $d-1\leq g-3$. 

    We assume by contradiction that there exists $p$ such that $\rr(E+p)\geq1$. Then $p$ must be contained in the same cycle of some $x_i$. 
    We start Dhar's algorithm from $\gamma$. Since $\rr(E)=0$ by (E1), but $\rr(E+p)\geq1$, the fire has to stop at $x_i,$ $p$ and possibly at some other $x_j$'s, with $j\in J\subset\{1,\dots, d-1\}\setminus\{i\}$. The edges at which the fire stops form a $k'$-edge cut, for some $k'\geq 2.$ We take the linearly equivalent divisor obtained by moving $x_i,p$ and $x_j$ with $j\in J$ along such a cut, towards the burnt direction, until the first point reaches a vertex: $x_i+p+\sum_{j\in J} x_j\sim \sum_{i=1}^{k'} y_i.$ By (E2) at most two points can reach a vertex. We assume that these points are $y_1$ and possibly $y_2$. Notice that the other $y_i$'s, with $i\neq 1,2$, again satisfy (E1),(E2). 
    
    We then consider another cycle $\gamma_1$, containing $y_1$ in the same connected component obtained by removing the $k'$-edge cut, containing $\gamma$.
    Either $\gamma_1$ contains another $x_{j'},$ for some $j'\notin \{i\}\cup J$ or it also contains $y_2$, which is a vertex. Otherwise, the linearly equivalent divisor containing $\sum_{i=1}^{k'} y_i$ in its support would be supported with multiplicity one in this cycle, and it would burn starting a fire at $\gamma_1$. 
    Let us discuss these two possibilities.
    \begin{enumerate}
    \item If the cycle $\gamma_1$ contains some $x_{j'}$; $j'\notin\{i\}\cup J,$ we repeat the above argument replacing $x_i,p, x_j$ with $j\in J$ with $x_{j'}$ and with the other points in the linearly equivalent divisor that together with $x_{j'}$ form a $k''$-edge cut, again determined by Dhar's burning algorithm when starting a fire at $\gamma.$
    By property (E2), the linearly equivalent divisor obtained by moving such points on this cut will now be supported at most at two vertices. We repeat this argument until we may assume that such vertices are in $\gamma$. If there is only one vertex, starting a fire at any other point in $\gamma$ would then burn everything since the other points satisfy (E1), (E2), giving a contradiction. 
    If the vertices are $2,$ we refer to next case.

    \item Let us assume that $y_1,y_2$ are both vertices, and moving them on incident empty cycles might give linearly equivalent divisors supported again on vertices. By property (E2), this happens only when we move again both vertices, along edges of the same length. 
    We continue by taking linearly equivalent divisors whose supports move towards $\gamma$ as above, until either we end up with an equivalent divisor where only one point moves to a vertex (and we refer to the previous case), or we have a linearly equivalent divisor whose support contains $z_1,z_2$ vertices on $\gamma$. 

    Let us now consider a pair of edges incident to $z_1,z_2$, which defines a $k$-edge cut, with possibly other edges containing some other points in the support of the divisor.
    
    If $k=2$: $z_1+z_2\sim z_1'+z_2'$, until one of the points reaches a vertex. In particular $\gamma$ is as in (NH1) and either both $z_1',z_2'$ cannot be vertices (as in (NH1)a), or one of then is uniquely contained in  $\gamma'$ (in (NH1)b). In both cases, starting Dhar's burning algorithm from $\gamma'$ (or $\gamma''$) burns the entire graph giving a contradiction.
    
    If $k>2$ then $\gamma$ satisfies (NH2) and we might assume that (at least) one of the edges of the cut, in $\gamma$, containing $z_1,z_2$ respectively (or two other vertices of $\gamma$ obtained again by linear equivalence), is contained uniquely in one of the cycles $\gamma'$ or $\gamma''$. It might be possible, for instance when the edge cut is not precisely the $k_0$-edge cut considered when defining $\gamma'$, that $\gamma'$ contains both $z_1,z_2$, but then $\gamma''$ does not by construction. Then we can start Dhar's burning algorithm from there, and the whole graph burns again.
    \end{enumerate}
\end{proof}

The argument in the proof of \cref{prp:no_sep} shows in particular why for non-hyperelliptic graphs, whose underlying graph is hyperelliptic, the equality for the Brill--Noether rank might be achieved only when $d=g-2+r$, (Section \ref{sc:Coppens}). Indeed, when $d\leq g-3+r$ we can construct an effective divisor of degree $d-1\leq g-3,$ whose support is not contained in a non-hyperelliptic cycle and two other cycles incident to it, in the two distinct components of the graph with the cycle removed. 
Similarly, for graphs with a cycle $\gamma$ satisfying (NH2), in order to have a cycle containing uniquely one of the vertices supported on $\gamma$, we need to define in general both $\gamma',\gamma''$.

On the other hand, when the underlying graph is non-hyperelliptic, but all cycles which are neither loops nor parallel edges determine a $2$-edge cut, we only need to consider the cycle $\gamma'$ in the above proof. This means that for such graphs, the same proof holds if we allow the degree of the divisor $E$ to be $d-1\leq g-2$.
\bibliographystyle{alpha}
\bibliography{bib}
\end{document}